\documentclass[11pt,a4paper]{article}
\usepackage{latexsym}
\usepackage{amsmath}
\usepackage{amssymb}
\usepackage{amscd}
\usepackage{amsthm}
\usepackage[all]{xy}

\newtheorem{thm}{Theorem}[section]

\newtheorem{lem}[thm]{Lemma}

\theoremstyle{definition}

\newtheorem{rem}[thm]{Remark}

\newcommand{\R}{\mathbb R}

\newcommand{\C}{\mathbb C}

\allowdisplaybreaks[4]

\newif\ifpdf \pdftrue
\ifx\pdfoutput\undefined \pdffalse \fi \ifx\pdfoutput\relax
\pdffalse \fi

\usepackage{makeidx}
\makeindex

\begin{document}

\title{A note on Taylor expansion of real funtion}

\author{Shun Tang}

\date{}

\maketitle

\vspace{-10mm}

\hspace{5cm}\hrulefill\hspace{5.5cm} \vspace{5mm}

\textbf{Abstract.} Let $f(x)$ be a real function which has $(n+1)$-th derivative on an interval $[a, b]$. For any point $x_0\in (a, b)$ and any integer $0\leq k\leq n$, denote by $S_{k,x_0}(x)$ the $k$-th truncation of the Taylor expansion of $f(x)$ at $x_0$, i.e. $$S_{k,x_0}(x)=\sum_{i=0}^k\frac{f^{(i)}(x_0)}{i!}(x-x_0)^i.$$ In this note, we consider the $L_2$-approximation of $f(x)$ by polynomials of degree $\leq k$, we show that $S_{k,x_0}(x)$ is the limit of the best approximations of $f(x)$ on $[x_0-\varepsilon, x_0+\varepsilon]$ as $\varepsilon\to 0$.

\textbf{2020 Mathematics Subject Classification:} 41A10, 41A52, 41A58


\section{Introduction}
Given a polynomial $P(x)$ with real coefficients, the $L_2$-distance between $f(x)$ and $P(x)$ in the space of continuous functions on $[a, b]$ is defined as 
$$\| f(x)-P(x)\|_2:=\Big( \int_a^b [f(x)-P(x)]^2 {\rm d}x \Big)^{1/2}.$$
Let $\R[x]_{\leq k}$ denote the set of real polynomials of degree less than or equal to $k$. It is a classical problem to find the best approximation of $f(x)$ in $\R[x]_{\leq k}$. To the  $L_2$-approximation of a periodic function by trigonometric polynomials, it is well known that the corresponding truncation of its Fourier expansion provides on one period the best approximation of any fixed degree, because the set of functions $\{\frac{1}{2}, \cos x, \sin x, \cdots, \cos kx, \sin kx, \cdots\}$ is an orthogonal basis with respect to the $L_2$ inner product. The Taylor expansion does not have similar property. For instance, consider the Taylor expansion of $f(x)=e^x$ at the point $x_0=0$, the constant term $e^0=1$ does not provide the best $L_2$-approximation of $e^x$ by degree-$0$ polynomials on any interval with central point $x_0=0$ e.g. on $[-1, 1]$ we have 
$$\Big( \int_{-1}^1 (e^x-\frac{e-1/e}{2})^2 {\rm d}x \Big)^{1/2}<\Big( \int_{-1}^1 (e^x-1)^2 {\rm d}x \Big)^{1/2}.$$
However, the Taylor expansion of a differentiable function has deep relation with its best local $L_2$-approximations by polynomials. We shall prove that when the closed neighborhood of $x_0$ gets smaller and smaller, on which the best $L_2$-approximation gets closer and closer to the truncation of the Taylor expansion. Our main theorem is the following.

\begin{thm}\label{A}
Let $f(x)$ be a real function which has $(n+1)$-th derivative on an interval $[a, b]$. Fix a point $x_0\in (a, b)$ and fix an integer $0\leq k\leq n$. Write the best $L_2$-approximation of $f(x)$ in $\R[x]_{\leq k}$ on a closed neighborhood $[x_0-\varepsilon, x_0+\varepsilon]\subset [a, b]$ with $\varepsilon>0$ in the form $$P_{k, \varepsilon}(x)=\sum_{i=0}^k a_{i, \varepsilon}(x-x_0)^i,$$ then we have $\lim\limits_{\varepsilon\to 0} a_{i, \varepsilon}=\frac{f^{(i)}(x_0)}{i!}$ for any $0\leq i\leq k$.
\end{thm}

The existence of the best $L_2$-approximation of $f(x)$ in $\R[x]_{\leq k}$ on any subinterval of $[a, b]$ is a well-known fact. The most convenient way to find them out is to transform the linear independent set $\{1, x, x^2, \cdots\}$ to an orthogonal basis via the Gram-Schmidt process with respect to the $L_2$ inner product. The resulting orthogonal polynomials are just the shifted Legendre polynomials up to a constant multiplication factor. In this note, we shall give a direct proof of Theorem~\ref{A} without using Legendre polynomials. 

\section{Proof of Theorem~\ref{A}}
We define $$J(a_{0, \varepsilon}, a_{1, \varepsilon}, \ldots, a_{k, \varepsilon})=\int_{x_0-\varepsilon}^{x_0+\varepsilon} [f(x)-\sum_{i=0}^k a_{i, \varepsilon}(x-x_0)^i]^2 {\rm d}x.$$ To find the potential extreme points of $J$ is to solve the following system of equations
\begin{equation}\label{e1}
\begin{cases}
\frac{\partial J}{\partial a_{0, \varepsilon}}=0\\
\frac{\partial J}{\partial a_{1, \varepsilon}}=0\\
\quad\vdots\\
\frac{\partial J}{\partial a_{k, \varepsilon}}=0.\\
\end{cases}
\end{equation}
We compute
\begin{align*}
\frac{\partial J}{\partial a_{i, \varepsilon}}&=-2\int_{x_0-\varepsilon}^{x_0+\varepsilon} [f(x)-\sum_{j=0}^k a_{j, \varepsilon}(x-x_0)^j](x-x_0)^i {\rm d}x\\
&=-2\int_{x_0-\varepsilon}^{x_0+\varepsilon}f(x)(x-x_0)^i {\rm d}x+2\int_{x_0-\varepsilon}^{x_0+\varepsilon}\sum_{j=0}^k a_{j, \varepsilon}(x-x_0)^{i+j} {\rm d}x\\
&=2\sum_{j=0}^k\frac{a_{j, \varepsilon}}{i+j+1}[\varepsilon^{i+j+1}-(-\varepsilon)^{i+j+1}]-2\int_{x_0-\varepsilon}^{x_0+\varepsilon}f(x)(x-x_0)^i {\rm d}x.
\end{align*}
Equation (1) is a system of linear equations of $k+1$ variables, we write it as a matrix equation $A_{k+1}\cdot X=W_{k+1}$. Then $A_{k+1}$ is a $(k+1)\times(k+1)$-matrix in which $a_{rs}=\frac{1}{r+s-1}[\varepsilon^{r+s-1}-(-\varepsilon)^{r+s-1}]$ and $W_{k+1}$ is a $(k+1)$-vector in which $w_r=\int_{x_0-\varepsilon}^{x_0+\varepsilon}f(x)(x-x_0)^{r-1} {\rm d}x$. For example, for $k=4$,
\begin{equation*}
A_5=\left(
\begin{array}{ccccc}
2\varepsilon & 0 & \frac{2}{3}\varepsilon^3 & 0 & \frac{2}{5}\varepsilon^5 \\
0 & \frac{2}{3}\varepsilon^3 & 0 & \frac{2}{5}\varepsilon^5 & 0 \\
\frac{2}{3}\varepsilon^3 & 0 & \frac{2}{5}\varepsilon^5 & 0 &  \frac{2}{7}\varepsilon^7 \\
0 & \frac{2}{5}\varepsilon^5 & 0 & \frac{2}{7}\varepsilon^7 & 0 \\
\frac{2}{5}\varepsilon^5 & 0 & \frac{2}{7}\varepsilon^7 & 0 & \frac{2}{9}\varepsilon^9 \\
\end{array}
\right).
\end{equation*}

Note that $A_{k+1}$ is also the Hessian matrix of $J$. The positive definiteness of $A_{k+1}$ follows from its inner-product expression $a_{rs}=\langle (x-x_0)^{r-1}, (x-x_0)^{s-1}\rangle_{L_2}$ and the fact that the set of functions $\{1, (x-x_0), (x-x_0)^2, \cdots, (x-x_0)^k\}$ is linear independent. In this note, we include an elementary proof of the positive definiteness of $A_{k+1}$, we shall directly prove ${\rm det}(A_{k+1})>0$ for all $k\geq 0$.

Firstly, set 
\begin{equation*}
\widetilde{A}_{k+1}:=\left(
\begin{array}{ccccc}
1 & 0 & \frac{1}{3} & \cdots & \frac{1-(-1)^{k+1}}{2(k+1)} \\
0 & \frac{1}{3} & 0 & \cdots & \frac{1-(-1)^{k+2}}{2(k+2)} \\
\frac{1}{3} & 0 & \frac{1}{5} & \cdots & \frac{1-(-1)^{k+3}}{2(k+3)}\\
\vdots & \vdots & \vdots & \ddots & \vdots \\
\frac{1-(-1)^{k+1}}{2(k+1)} & \frac{1-(-1)^{k+2}}{2(k+2)} & \frac{1-(-1)^{k+3}}{2(k+3)} & \cdots & \frac{1-(-1)^{2k+1}}{2(2k+1)} \\
\end{array}
\right),
\end{equation*}
one computes 
\begin{align*}
{\rm det}(A_{k+1})&=\sum_{\sigma\in S_{k+1}}(-1)^{{\rm sgn}(\sigma)}a_{1\sigma(1)}a_{2\sigma(2)}\cdots a_{(k+1)\sigma(k+1)}\\
&=2^{k+1}\varepsilon^{\Big(\sum_{l=1}^{k+1}\big(l+\sigma(l)-1\big)\Big)}{\rm det}(\widetilde{A}_{k+1})\\
&=2^{k+1}\varepsilon^{\big(2(\sum_{l=1}^{k+1}l)-(k+1)\big)}{\rm det}(\widetilde{A}_{k+1})\\
&=2^{k+1}\varepsilon^{(k+1)^2}{\rm det}(\widetilde{A}_{k+1})
\end{align*}
where $S_{k+1}$ denotes the $(k+1)$-th symmetric group. Since $\varepsilon>0$, we only need to prove that ${\rm det}(\widetilde{A}_{k+1})>0$ for all $k\geq 0$.

Set
\begin{equation*}
u=
\begin{cases}
\frac{k}{2}, & \text{ if } k \text{ is even} \\
\frac{k+1}{2}, & \text{ if } k \text{ is odd}\\
\end{cases},
\quad\quad
v=
\begin{cases}
\frac{k}{2}+1, & \text{ if } k \text{ is even} \\
\frac{k+1}{2}, & \text{ if } k \text{ is odd}\\
\end{cases}
\end{equation*} 
and choose a permutation $\sigma\in S_{k+1}$ which sends $(1, 2, \cdots, k+1)$ to $(2, 4, \cdots, 2u, 1, 3, \cdots, 2v-1)$. 

Apply the elementary row operations corresponding to $\sigma$ to the matrix $\widetilde{A}_{k+1}$ and apply the elementary column operations corresponding to $\sigma$ to the resulting matrix, $\widetilde{A}_{k+1}$ is then transformed to a partitioned matrix
\begin{equation*}
\left(
\begin{array}{cc}
B_u & 0 \\
0 & C_v\\
\end{array}
\right)
\end{equation*}
whose determinant is equal to ${\rm det}(\widetilde{A}_{k+1})$. The block $B_u$ is a $u\times u$-matrix, $C_v$ is a $v\times v$-matrix. They are of the following forms
\begin{equation*}
B_u=
\left(
\begin{array}{cccc}
\frac{1}{3} & \frac{1}{5} & \cdots & \frac{1}{2u+1}\\
\frac{1}{5} & \frac{1}{7} & \cdots & \frac{1}{2u+3}\\
\vdots & \vdots & \ddots & \vdots \\
\frac{1}{2u+1} & \frac{1}{2u+3} & \cdots & \frac{1}{4u-1}\\
\end{array}
\right),
\quad\quad
C_v=
\left(
\begin{array}{cccc}
1 & \frac{1}{3} & \cdots & \frac{1}{2v-1}\\
\frac{1}{3} & \frac{1}{5} & \cdots & \frac{1}{2v+1}\\
\vdots & \vdots & \ddots & \vdots \\
\frac{1}{2v-1} & \frac{1}{2v+1} & \cdots & \frac{1}{4v-3}\\
\end{array}
\right).
\end{equation*}
If we can prove ${\rm det}(B_u)>0$ and ${\rm det}(C_v)>0$ for all $u$ and all $v$, then we get ${\rm det}(\widetilde{A}_{k+1})>0$ for all $k\geq 0$. This can be seen from the following lemma.

\begin{lem}\label{B}
For any integer $i\geq 1$ and any integer $t\geq 0$, the determinant of the following matrix 
\begin{equation*}
D(i, t)=
\left(
\begin{array}{ccccc}
\frac{1}{i} & \frac{1}{i+2} & \frac{1}{i+4} & \cdots & \frac{1}{i+2t}\\
\frac{1}{i+2} & \frac{1}{i+4} & \frac{1}{i+6} & \cdots & \frac{1}{i+2(t+1)}\\
\frac{1}{i+4} & \frac{1}{i+6} & \frac{1}{i+8} & \cdots & \frac{1}{i+2(t+2)}\\
\vdots & \vdots & \vdots& \ddots & \vdots \\
\frac{1}{i+2t} & \frac{1}{i+2(t+1)} & \frac{1}{i+2(t+2)} & \cdots & \frac{1}{i+4t}\\
\end{array}
\right)
\end{equation*}
is always positive.
\end{lem}
\begin{proof}
For any integer $i\geq 1$, ${\rm det}\big(D(i, 0)\big)=\frac{1}{i}>0$ and ${\rm det}\big(D(i, 1)\big)=\frac{1}{i(i+4)}-\frac{1}{(i+2)^2}>0$. For general $t\geq 2$, we firstly multiply the $j$-th row of $D(i, t)$ by a constant $i+2(j-1)>0$ from $j=1$ to $j=t+1$, then we get a matrix
\begin{equation*}
\left(
\begin{array}{ccccc}
1 & \frac{i}{i+2} & \frac{i}{i+4} & \cdots & \frac{i}{i+2t}\\
1 & \frac{i+2}{i+4} & \frac{i+2}{i+6} & \cdots & \frac{i+2}{i+2(t+1)}\\
1 & \frac{i+4}{i+6} & \frac{i+4}{i+8} & \cdots & \frac{i+4}{i+2(t+2)}\\
\vdots & \vdots & \vdots& \ddots & \vdots \\
1 & \frac{i+2t}{i+2(t+1)} & \frac{i+2t}{i+2(t+2)} & \cdots & \frac{i+2t}{i+4t}\\
\end{array}
\right)
\end{equation*}
whose determinant has the same sign as ${\rm det}\big(D(i, t)\big)$. Next, we multiply the $j$-th column of this matrix by a constant $i+2(j-1)>0$ from $j=2$ to $j=t+1$, then we get the following matrix
\begin{equation*}
\left(
\begin{array}{ccccc}
1 & i & i & \cdots & i\\
1 & \frac{(i+2)^2}{i+4} & \frac{(i+2)(i+4)}{i+6} & \cdots & \frac{(i+2)(i+2t)}{i+2(t+1)}\\
1 & \frac{(i+2)(i+4)}{i+6} & \frac{(i+4)^2}{i+8} & \cdots & \frac{(i+4)(i+2t)}{i+2(t+2)}\\
\vdots & \vdots & \vdots& \ddots & \vdots \\
1 & \frac{(i+2)(i+2t)}{i+2(t+1)} & \frac{(i+4)(i+2t)}{i+2(t+2)} & \cdots & \frac{(i+2t)^2}{i+4t}\\
\end{array}
\right)
\end{equation*}
whose determinant also has the same sign as ${\rm det}\big(D(i, t)\big)$. 

Adding a $-1$ multiple of the first row to the $j$-th row from $j=2$ to $j=t+1$, we get 
\begin{equation*}
\left(
\begin{array}{ccccc}
1 & i & i & \cdots & i\\
0 & \frac{4}{i+4} & \frac{8}{i+6} & \cdots & \frac{4t}{i+2(t+1)}\\
0 & \frac{8}{i+6} & \frac{16}{i+8} & \cdots & \frac{8t}{i+2(t+2)}\\
\vdots & \vdots & \vdots& \ddots & \vdots \\
0 & \frac{4t}{i+2(t+1)} & \frac{8t}{i+2(t+2)} & \cdots & \frac{4t^2}{i+4t}\\
\end{array}
\right)
\end{equation*}
so that the sign of ${\rm det}\big(D(i, t)\big)$ is the same as the sign of the determinant of 
\begin{equation*}
\left(
\begin{array}{cccc}
\frac{4}{i+4} & \frac{8}{i+6} & \cdots & \frac{4t}{i+2(t+1)}\\
\frac{8}{i+6} & \frac{16}{i+8} & \cdots & \frac{8t}{i+2(t+2)}\\
\vdots & \vdots& \ddots & \vdots \\
\frac{4t}{i+2(t+1)} & \frac{8t}{i+2(t+2)} & \cdots & \frac{4t^2}{i+4t}\\
\end{array}
\right),
\end{equation*}
and hence is the same as the sign of the determinant of 
\begin{equation*}
\left(
\begin{array}{cccc}
\frac{1}{i+4} & \frac{2}{i+6} & \cdots & \frac{t}{i+2(t+1)}\\
\frac{1}{i+6} & \frac{2}{i+8} & \cdots & \frac{t}{i+2(t+2)}\\
\vdots & \vdots& \ddots & \vdots \\
\frac{1}{i+2(t+1)} & \frac{2}{i+2(t+2)} & \cdots & \frac{t}{i+4t}\\
\end{array}
\right).
\end{equation*}

Multiplying the $j$-th column of the above matrix by $\frac{1}{j}$ from $j=2$ to $j=t$, we may conclude that the sign of ${\rm det}\big(D(i, t)\big)$ is the same as the sign of ${\rm det}\big(D(i+4, t-1)\big)$. Repeating these manipulations, we know that ${\rm det}\big(D(i, t)\big)$ has the same sign as ${\rm det}\big(D(i+4(t-1), 1)\big)$ which is positive. So we are down.
\end{proof}
 
\begin{rem}\label{B+}
The matrices $D(i, t)$ are special examples of Cauchy matrix $\big(\frac{1}{a_p+b_q}\big)_{(t+1)\times (t+1)}$ where $a_p=i+2(p-1)$ and $b_q=2(q-1)$. From the calculation formula of the determinant of Cauchy matrix 
$${\rm det}\big(\frac{1}{a_p+b_q}\big)=\frac{\prod\limits_{1\leq p<q\leq t+1}(a_p-a_q)(b_p-b_q)}{\prod\limits_{1\leq p,q\leq t+1}(a_p+b_q)},$$
one easily sees ${\rm det}\big(D(i, t)\big)>0$.
\end{rem}

To compute the limit of the solutions of $A_{k+1}\cdot X=W_{k+1}$ as $\varepsilon\to 0$, we describe the elements in $A_{k+1}^{-1}$.

\begin{lem}\label{C}
For any $1\leq r, s\leq k+1$, the element in $A_{k+1}^{-1}$ at the place $(r, s)$ is of the form $\alpha_{rs}\cdot (\frac{1}{\varepsilon})^{r+s-1}$ where $\alpha_{rs}$ is a real number depending only on $k$. If $r+s$ is an odd number, then $\alpha_{rs}=0$.
\end{lem}
\begin{proof}
Denote by $A_{k+1}^*$ the adjugate matrix of $A_{k+1}$. The first statement is easily seen from $A_{k+1}^{-1}=\frac{1}{{\rm det}(A_{k+1})}A_{k+1}^*$,  ${\rm det}(A_{k+1})=2^{k+1}\varepsilon^{(k+1)^2}{\rm det}(\widetilde{A}_{k+1})$ and the fact that the cofactor of $a_{sr}$ equals
$$\sum_{\substack{\sigma\in S_{k+1}\\ \sigma(s)=r}}(-1)^{{\rm sgn}(\sigma)}a_{1\sigma(1)}a_{2\sigma(2)}\cdots a_{(s-1)\sigma(s-1)}a_{(s+1)\sigma(s+1)}\cdots a_{(k+1)\sigma(k+1)}$$
which is $\varepsilon^{(k+1)^2-(s+r-1)}$ multiplied by a real number. Furthermore, if $r+s$ is an odd number, then each term in the above summation must contain some $a_{s'r'}$ as a factor such that $s'+r'$ is also odd. This implies that $\alpha_{rs}=0$.
\end{proof}

Since $f(x)$ has $(n+1)$-th derivative on $[a, b]$, for any $0\leq k\leq n$ we may write
$$f(x)=f(x_0)+f'(x_0)(x-x_0)+\cdots+\frac{1}{(k+1)!}f^{(k+1)}(x_0)(x-x_0)^{k+1}+o\big((x-x_0)^{k+1}\big).$$
Set 
$$R_k(x):=f(x)-\big[f(x_0)+f'(x_0)(x-x_0)+\cdots+\frac{1}{k!}f^{(k)}(x_0)(x-x_0)^{k}\big],$$
then 
$$\lim_{x\to x_0}\frac{R_k(x)}{(x-x_0)^{k+1}}=\frac{1}{(k+1)!}f^{(k+1)}(x_0).$$ 

We define a new function
\begin{equation*}
G_k(x)=
\begin{cases}
\frac{R_k(x)}{(x-x_0)^{k+1}}, & \text{ if } x\neq x_0 \\
\frac{1}{(k+1)!}f^{(k+1)}(x_0), & \text{ if } x=x_0 \\
\end{cases}.
\end{equation*}
It is a continuous function and it satisfies $G_k(x)(x-x_0)^{k+1}=R_k(x)$. For $1\leq s\leq k+1$, \begin{align*}
w_s&=\int_{x_0-\varepsilon}^{x_0+\varepsilon}f(x)(x-x_0)^{s-1}{\rm d}x\\
&=\sum_{r=1}^{k+1}\int_{x_0-\varepsilon}^{x_0+\varepsilon}\frac{1}{(r-1)!}f^{(r-1)}(x_0)(x-x_0)^{r+s-2}{\rm d}x+\int_{x_0-\varepsilon}^{x_0+\varepsilon}G_k(x)(x-x_0)^{s+k}{\rm d}x\\
&=\sum_{r=1}^{k+1}\bigg\{\frac{1}{(r-1)!}f^{(r-1)}(x_0)\frac{1}{r+s-1}[\varepsilon^{r+s-1}-(-\varepsilon)^{r+s-1}]\bigg\}+\int_{x_0-\varepsilon}^{x_0+\varepsilon}G_k(x)(x-x_0)^{s+k}{\rm d}x\\
&=\sum_{r=1}^{k+1}\frac{1}{(r-1)!}f^{(r-1)}(x_0)a_{rs}+\int_{x_0-\varepsilon}^{x_0+\varepsilon}G_k(x)(x-x_0)^{s+k}{\rm d}x.
\end{align*}
Denote the unique solution of $A_{k+1}\cdot X=W_{k+1}$ by $(a_{0,\varepsilon}, a_{1,\varepsilon}, \cdots, a_{k,\varepsilon})^{\rm T}$, then for any $0\leq i\leq k$,
$$a_{i,\varepsilon}-\frac{1}{i!}f^{(i)}(x_0)=\sum_{s=1}^{k+1}\bigg\{\alpha_{(i+1)s}(\frac{1}{\varepsilon})^{s+i}\int_{x_0-\varepsilon}^{x_0+\varepsilon}G_k(x)(x-x_0)^{s+k}{\rm d}x\bigg\}.$$

We are left to estimate the absolute value of $a_{i,\varepsilon}-\frac{1}{i!}f^{(i)}(x_0)$. Since $G_k(x)$ is continuous and hence is bounded over $[a, b]$, we may suppose that $\mid G_k(x)\mid \leq M$ over $[a, b]$ for some positive number $M$. Then 
\begin{align*}
\mid a_{i,\varepsilon}-\frac{1}{i!}f^{(i)}(x_0)\mid &\leq \sum_{s=1}^{k+1}\mid \alpha_{(i+1)s}\mid\cdot M\cdot (\frac{1}{\varepsilon})^{s+i}\int_{x_0-\varepsilon}^{x_0+\varepsilon}\mid (x-x_0)^{s+k}\mid{\rm d}x\\
&=\sum_{s=1}^{k+1}\mid \alpha_{(i+1)s}\mid\cdot 2M\cdot (\frac{1}{\varepsilon})^{s+i}\int_{x_0}^{x_0+\varepsilon} (x-x_0)^{s+k}{\rm d}x\\
&=\sum_{s=1}^{k+1}\mid \alpha_{(i+1)s}\mid\cdot \frac{2M}{s+k+1}\cdot \varepsilon^{k+1-i}
\end{align*}
which tends to $0$ as $\varepsilon\to 0$. This completes the proof of Theorem~\ref{A}.

\begin{rem}\label{D}
(a). The relation between Taylor expansion and the best $L_2$-approximation by polynomials is much clearer for analytic complex functions. For example, assume that $f$ is a complex function which is analytic in the open unit disk $D: \mid z\mid<1$ and continuous on $\overline{D}: \mid z\mid\leq 1$, then for any $k\geq 0$, the $k$-th truncation of the Taylor expansion of $f$ at $z_0=0$ is the best $L_2$-approximation of $f$ on the unit circle $C: \mid z\mid=1$. That is 
$$\int_C \mid f(z)-S_k(z)\mid^2\mid {\rm d}z\mid\leq \int_C \mid f(z)-P(z)\mid^2\mid {\rm d}z\mid,\qquad \forall P(z)\in \C[z]_{\leq k},$$
where $S_k(z)=\sum_{i=0}^k\frac{f^{(i)}(0)}{i!}z^i$ and $\C[z]_{\leq k}$ is the set of complex polynomials of degree at most $k$. This is simply because $\{1, z, z^2, \cdots, z^k\}$ is an orthogonal basis with respect the $L_2$ hermitian inner product on $C$, i.e.
$$\langle z^p, z^q\rangle=\int_C z^p\overline{z}^q\mid {\rm d}z\mid=0,\qquad \text{ if }p\neq q.$$
This statement is also true if the line integral (integral on $C$) is replaced by a surface integral (integral on $\overline{D}$), see \cite[Section 6.1, 6.5]{Wal}.

(b). Replacing the $L_2$-norm by the $L_\infty$-norm
$$\| f(x)-P(x)\|_\infty:=\max_{a\leq x\leq b} \mid f(x)-P(x)\mid,$$
one can ask whether the local Chebyshev best consistent approximations of $f(x)$ in $\R[x]_{\leq k}$ around $x_0\in (a, b)$ also converge to the $k$-th truncation of the Taylor expansion of $f(x)$ at $x_0$. The answer is YES, because the Chebyshev best consistent approximating polynomial in $\R[x]_{\leq k}$ is a Lagrange interpolating (and also a Newton interpolating) polynomial of degree $k$. As the interpolating interval gets smaller and smaller, the difference quotients appearing in the interpolating polynomials converge to the Taylor coefficients.

(c). The main theorem in this note originates in a simple question from the author's teaching, whether the Taylor coefficients can be explained by the idea of least square approximation. Different from the least square approximation on complex plane, up to now, the author has not found any relevant conclusion or expression about real functions in previous literature. So the author wrote down this note and shared his findings, hoping that the teacher who is teaching related contents of Taylor expansion, could help students understand the significance of Taylor coefficients from the perspective of least square approximation.

(d). Although the author knows that the Taylor polynomial is the limit of the local best $L_2$-approximations (resp. $L_\infty$-approximations), he can not answer the following question: given any $P(x)\in \R[x]_{\leq k}$ which is not equal to the $k$-th Taylor polynomial $S_{k, x_0}(x)$, if there exists a positive number $\delta$ such that for all $0<\varepsilon<\delta$,
$$\Big( \int_{x_0-\varepsilon}^{x_0+\varepsilon} [f(x)-S_{k, x_0}(x)]^2 {\rm d}x \Big)^{1/2}<\Big( \int_{x_0-\varepsilon}^{x_0+\varepsilon} [f(x)-P(x)]^2 {\rm d}x \Big)^{1/2}$$
$$\Big( \text{ resp. }  \max_{{x_0-\varepsilon}\leq x\leq {x_0+\varepsilon}} \mid f(x)-S_{k, x_0}(x)\mid<\max_{{x_0-\varepsilon}\leq x\leq {x_0+\varepsilon}} \mid f(x)-P(x)\mid\Big).$$
\end{rem}

\hspace{5cm} \hrulefill\hspace{5.5cm}

Shun Tang

Beijing Advanced Innovation Center for Imaging Theory and Technology

Academy for Multidisciplinary Studies, Capital Normal University

School of Mathematical Sciences, Capital Normal University

West 3rd Ring North Road 105, 100048 Beijing, P. R. China

E-mail: shun.tang@outlook.com

\end{document}